\documentclass[12pt,reqno]{amsart}
\usepackage{graphicx} % Required for inserting images
\usepackage{fancyhdr}
\usepackage{latexsym, amsmath, enumerate, color, amssymb,amsthm, amsfonts, hyperref,mathtools, rotating, amscd, enumerate}

\newtheorem{theorem}{Theorem}
\newtheorem{corollary}[theorem]{Corollary}
\newtheorem{lemma}[theorem]{Lemma}
\newtheorem{prop}[theorem]{Proposition}
\theoremstyle{definition} % switch to 'definition' theorem style

\newtheorem{example}[theorem]{Example}
\newtheorem{remark}[theorem]{Remark}
\title{Generalized Frobenius Number of Three Variables}
\author{Kittipong Subwattanachai}
\address{Graduate School of Mathematics, Nagoya University, Nagoya, Japan.}
\email{subwattanachai.k@gmail.com}

\date{\today}

\begin{document}
	
	\maketitle
	
	\begin{abstract}
		For $ k \geq 2 $,   we let $ A = (a_{1}, a_{2}, \ldots, a_{k}) $ be a $k$-tuple of positive integers with $\gcd(a_{1}, a_2, \ldots, a_k) =1$ and, for a non-negative integer $s$, the generalized Frobenius number of $A$, $g(A;s) = g(a_1, a_2, \ldots, a_k;s)$, the largest integer that has at most $s$ representations in terms of $a_1, a_2, \ldots, a_k$ with non-negative integer coefficients.
		In this article, we give a formula for the generalized Frobenius number of three positive integers $(a_1,a_2,a_3)$ with certain conditions.
	\end{abstract}
	
	\section{Introduction}
	For a positive integer $ k \geq 2 $, $ A = (a_{1}, a_{2}, \ldots, a_{k}) $ is a $k$-tuple of positive integers with $ \gcd(A) = \gcd(a_{1}, a_2, \ldots, a_k) =1$.
	If we let $\mathrm{R}(A) = \mathrm{R}(a_{1},\ldots,a_{k}) = \{x_{1}a_{1} + \cdots + x_{k}a_{k} \mid x_{i}\in \mathbb{Z}_{\geq 0}, i =1,2,\ldots,k\}$ be the set of integers representable as non-negative linear combinations of $a_{1}, \ldots, a_{k}$ and let $\mathrm{NR}(A)  = \mathrm{NR}(a_{1},a_{2}, \ldots, a_{k}) = \mathbb{Z}_{\geq-1}\setminus\mathrm{R}(A)$ be the set of integers not representable as non-negative integer combinations of $a_{1}, a_{2}, \ldots, a_{k}$. It is known that $\mathrm{NR}(A)$ is finite if and only if $\gcd(a_{1}, a_{2}, \ldots, a_{k}) =1$, see for example \cite{Rosales-Numerical Semigroup}.
	
	There is the well-known linear Diophantine problem, posed by Sylvester \cite{Sylvester_MathQuestion}, known as the \emph{Frobenius problem}\footnote{It is also known as the coin problem, postage stamp problem, or Chicken McNugget problem, involves determining the largest value that cannot be formed using only coins of specified denominations.}: Given positive integers $a_{1}, a_{2}, \ldots, a_{k}$ such that $\gcd(a_{1}, a_{2}, \ldots, a_{k}) = 1$, find the largest integer that \emph{cannot} be expressed as a non-negative integer linear combination of these numbers.
	\emph{The largest integer} is called the \emph{Frobenius number} of the tuple $ A = (a_{1}, a_{2}, \ldots, a_{k})$, and is denoted by $g(A) = g(a_{1}, a_{2}, \ldots, a_{k})$. 
	With the above notation, the Frobenius number is given by
	\begin{equation*}
		g(A) = \max \mathrm{NR}(A).
	\end{equation*}
	Note that if all non-negative integers can be expressed as a non-negative integer linear combination of $A$, then $g(A)= -1$. For example, $g(1,2) = -1$.
	
	For two coprime positive integers $A = (a,b)$, it is shown by Sylvester \cite{Sylvester-On_subinvariants} that
	\begin{equation}\label{eq:twovariableclassical}
		g(a,b) = ab-a-b.
	\end{equation}
	For example, consider $A = (a,b) = (3,5)$. Then the Frobenius number of $A$ is given by $g(3,5) = 15-3-5 = 7$, which means that all integers $n > 7$ can be expressed as a non-negative integer linear combination of $3$ and $5$. 
	
	%Explicit but complicated formulas for the Frobenius number in three variables are given by Tripathi \cite{Tripathi-Formulae_for_Frobenius_three_variables}. Though closed forms for general case are hopeless for $k \geq 4$, several formulas for Frobenius numbers have been considered for special cases.
	Tripathi \cite{Tripathi-Formulae_for_Frobenius_three_variables} has provided explicit but complicate formulas for calculating the Frobenius number in three variables. However, it is important to note that closed-form solutions for the general case become increasingly challenging when the number of variables exceeds three $(k > 3)$. Nevertheless, various formulas have been proposed for Frobenius numbers in specific scenarios or special cases.
	For example, explicit formulas in some particular cases of sequences, including arithmetic, geometric-like, Fibonacci, Mersenne, and triangular (see \cite{Roble_Rosales} and references therein) are known.
	
	%For three variables, Thipathi [??] showed the formulae for the Frobenius numbers in three variables in 2017.
	%\begin{center}
	%    Thipathi results....
	%\end{center}
	
	For a given positive integer $n$, we let $d(n;A) = d(n;a_{1}, a_{2}, \ldots, a_{k})$ be the number of representations to $a_{1}x_{1} + a_{2}x_{2} + \cdots + a_{k}x_{k} = n$. Its generating series is given by
	\begin{equation*}
		\sum_{n\geq 0} d(n;a_1,\dots,a_k) x^n = \frac{1}{(1-x^{a_{1}})(1-x^{a_{2}})\cdots(1-x^{a_{k}}) }.
	\end{equation*}
	Sylvester \cite{Sylvester-On_the_partition_of_numbers} and Cayley \cite{Cayley:double_partition} show that $d\left(n ; a_1, a_2, \ldots, a_k\right)$ can be expressed as the sum of a polynomial in $n$ of degree $k-1$ and a periodic function of period $a_1 a_2 \cdots a_k$. Using Bernoulli numbers, Beck, Gessel, and Komatsu \cite{Beck_Gessel_Komatsu} derive the explicit formula for the polynomial section. Tripathi \cite{Tripathi-The_number_of_solutions} provides a formula for $d(n;a_{1},a_{2})$. Komatsu \cite{Komatsu-On_the_number_of_solutions} shows that the periodic function part is defined in terms of trigonometric functions for three variables in the pairwise coprime case.
	%However, with greater $a_1, a_2, a_3$, the calculation becomes highly challenging. 
	Binner \cite{Binner-number of solutions to axby+cz=n} provides a formula for the number of non-negative integer solutions to the equation $ax + by + cz = n$ and finds a relationship between the number of solutions and quadratic residues.\\

    	In this work, we will focus on a generalization of the Frobenius number. For a given non-negative integer $s$, let $$g(A;s) = g(a_1, a_2, \ldots, a_k;s) = \max\{ n \in \mathbb{Z} \mid d(n;A) \leq s\}$$ be the largest integer such that the number of expressions that can be represented by $a_1, a_2, \ldots, a_k$ is at most $s$. Notice that  $g(a_1, a_2, \ldots, a_k)= g(a_1, a_2, \ldots, a_k;0)$. That means all integers bigger than $g(A;s)$ have at least $s+1$ representations. The $g(A;s)$ is called \emph{the generalized Frobenius number}. Furthermore, $g(A;s)$ is well-defined (i.e. bounded above) (see \cite{Fukshansky-Schurmann-Bounds_on})%Furthermore, some authors study in the $s$\emph{-Frobenius number} of the set $A$, defined to be the largest integer $g'(A ; s)$, such that the number of expressions is \emph{exactly} $s$ (see \cite{Beck_Kifer,Fukshansky-Schurmann-Bounds_on}). If $s=0$, then $g(A;0) = g(A)$ becomes the original Frobenius number. However,  in some circumstances, any positive integer does not have precisely $s$ representations, especially as $s$ becomes larger (see \cite[Remark 4]{Komatsu-triangular}).  As a result, it is better to treat $g(A;s)$ rather than $g'(A;s)$. Clearly, if we know $g(A;s)$ and $g(A;s+1)$, then we can also get $g'(A;s)$. Fukshansky and Schurmann \cite{Fukshansky-Schurmann-Bounds_on} provides upper and lower bounds on the $s$-Frobenius number for any non-negative integer $s$ using techniques from the geometry of numbers. So, it implies that the generalized Frobenius number $g(A;s)$ is finite since $g'(A;s) \leq g(A;s) < g'(A;s+1)$.
    	
    	As a generalization of \eqref{eq:twovariableclassical}, for $A = (a,b)$ and $s \in\mathbb{Z}_{\geq0}$, (see \cite{Beck_Robins}), an exact formula for $g(A,s) = g(a,b;s)$ is given by 
    	\begin{equation}
    		g(a,b;s) = (s+1)ab-a-b.    \label{g(a,b;p)}
    	\end{equation}
    	
    	In general, we have $d\big( g(A;s);A \big) \leq s$, but in the case $|A| = 2$ followed by Theorem 1 in Beck and Robins \cite{Beck_Robins} one can show that actually $d\big( g(A;s);A \big) = s$. Similar to the $s=0$ case, exact formulas for the generalized Frobenius number in the cases $k\geq 3$ are still unknown. For $k=3$ exact formulas are just known for special cases. For example, there are explicit results in the case of triangular numbers \cite{Komatsu-triangular}, repunits \cite{Komatsu-The Frobenius number repunits} and Fibonacci numbers \cite{Komatsu-Ying-Frob_Fibonacci}. Recently, Binner \cite{Binner-binner2021bounds} provide bounds for the number of solutions $a_{1}x_{1}+ a_{2}x_{2} + a_{3}x_{3} = n$ and use these bounds to solve $g(a_{1},a_{2},a_{3};s)$ when $s$ is large.  In 2022, Woods \cite{Woods-woods2022generalized} provide formulas and asymptotics for the generalized Frobenius problem using the restricted partition function.
	
	Although obtaining a general formula for three or more variables might be challenging or even impossible, this work presents Theorem \ref{Main-Thm}, which significantly generalizes the results of Komatsu \cite{Komatsu-triangular} considering three variables with certain divisibility conditions. 
 % Our main result is the following explicit formula for a special case of the generalized Frobenius number in three variables. 
	
	% is to show the formula for the generalized Frobenius number of three positive integers $(a,b,c) = (q_{1}q_{2}, q_{2}q_{3}, q_{3}q_{4})$ such that $\gcd(a,b,c) =1$, $\gcd(a,b) = q_{2}$, and $\gcd(b,c) = q_{3}$. 
	\begin{theorem}\label{Main-Thm}
		Let $a_{1}, a_{2}$ and $a_{3}$ be positive integers with $\gcd(a_{1},a_{2},a_{3}) = 1$ and let $t\in\mathbb{Z}_{\geq0}$. 
		If $d_{1} = gcd(a_{2},a_{3})$ and suppose that $a_{1} \equiv 0\pmod{\frac{a_{2}}{d_{1}}}$ or $a_{1} \equiv 0\pmod{\frac{a_{3}}{d_{1}}}$, then
		\begin{equation*}
			g\bigg( 
			a_{1}, a_{2}, a_{3}; \sum_{j=0}^{t}\left\lceil\frac{j a_{2}a_{3}}{a_{1}d_{1}^{2}} \right\rceil \bigg)
			=
			(t+1)\frac{a_{2}a_{3}}{d_{1}} + a_{1}d_{1} - a_{1} - a_{2} - a_{3}.
		\end{equation*}
	\end{theorem}
		\begin{remark}\label{remark: after thm1}
		In Theorem \ref{Main-Thm}, the order of integers in a tuple $(a_{1},a_{2},a_{3})$ is not necessary due to symmetry of $g$. So, if $d_{2} = \gcd(a_{1},a_{3})$ and $a_{2} \equiv 0\pmod{\frac{a_{1}}{d_{2}}}$ or $a_{2} \equiv 0\pmod{\frac{a_{3}}{d_{2}}}$, then
		\begin{equation*}
			g\bigg( 
			a_{1}, a_{2}, a_{3}; \sum_{j=0}^{t}\left\lceil\frac{j a_{1}a_{3}}{a_{2}d_{2}^{2}} \right\rceil \bigg)
			=
			(t+1)\frac{a_{1}a_{3}}{d_{2}} + a_{2}d_{2} - a_{1} - a_{2} - a_{3}.
		\end{equation*}	
		Similarly, if $d_{3} = \gcd(a_{1},a_{2})$ and $a_{3} \equiv 0\pmod{\frac{a_{1}}{d_{3}}}$ or $a_{3} \equiv 0\pmod{\frac{a_{2}}{d_{3}}}$, then
		\begin{equation*}
			g\bigg(	a_{1}, a_{2}, a_{3}; \sum_{j=0}^{t}\left\lceil\frac{j a_{1}a_{2}}{a_{3}d_{1}^{2}} \right\rceil \bigg)	=
			(t+1)\frac{a_{1}a_{2}}{d_{3}} + a_{3}d_{3} - a_{1} - a_{2} - a_{3}.
		\end{equation*}
		\end{remark}
	\begin{remark}
		Notice that Theorem \ref{Main-Thm} may not cover all cases of $g(a_1,a_2,a_3;s)$ for all $s \in\mathbb{Z}_{\geq 0}$. If we let
		\begin{equation*}
			\mathbf{U}_{(a_{1},a_{2},a_{3})} := \bigcup_{i=1}^{3}\Bigg\{ \sum_{j=0}^{t}\left\lceil\frac{j \prod_{\substack{1 \leq \ell  \leq 3 \\  \ell \neq i}}a_{\ell}}{a_{i}d_{i}^{2}} \right\rceil  \mid t \geq 0 \Bigg\} 
			%			\Bigg\{ \sum_{j=0}^{s}\left\lceil\frac{j a_{2}a_{3}}{a_{1}d_{1}^{2}} \right\rceil  \mid s \geq 0 \Bigg\} 
			%			\cup
			%			\Bigg\{ \sum_{j=0}^{s}\left\lceil\frac{j a_{1}a_{3}}{a_{2}d_{2}^{2}} \right\rceil  \mid s \geq 0 \Bigg\} 
			%			\cup
			%			\Bigg\{ \sum_{j=0}^{s}\left\lceil\frac{j a_{1}a_{2}}{a_{3}d_{1}^{2}} \right\rceil  \mid s \geq 0 \Bigg\} 
			\subseteq \mathbb{Z}_{\geq 0}, 
            % \{ d(n;a_1,a_2,a_3) \mid n\in\mathbb{Z}_{>0}\}.	
		\end{equation*}
		 and consider $(a_{1},a_{2},a_{3}) = (10,15,21)$ as in Example \ref{ex: triangular numbers}.
		Since $d(120;10,15,21) = 6$, but 
		$$ 6 \not\in\mathbf{U}_{(10,15,21)}  = \{0, 1, 2, 3, 4, 5, 7, 9, 11, 14, 17, 20, 22, 24,\ldots\}.
        % \subsetneq \{ d(n;10,15,21) \mid n\in\mathbb{Z}_{>0}\}.
        $$
		And another example, if $a_{1}, a_{2}$ and $a_{3}$ are of the form in Corollary \ref{Cor-2<=m1 < m2 < m3}, then we obtain that
		\begin{equation*}
			\mathbf{U}_{(a_{1},a_{2},a_{3})} = \{1,3,6,10,15,21,28,36,45,55,66,78, \ldots\} = \{t_{k} \mid k\in\mathbb{Z}_{\geq 0} \},
		\end{equation*}
		where $t_{k}$ is the $k$th triangular number which is given by $t_{k} = \binom{k+1}{2}$.
	\end{remark}
	We will give the proof in Section \ref{Section-ProofThm1}. Then, in Section \ref{section: special cases}, we discuss some special cases of Theorem \ref{Main-Thm} and relate them to previous study.

	%	Recall that the objective of this paper is shown the explicit formulas of the generalized Frobenius number of three consecutive triangular number $g(t_{n},t_{n+1}, t_{n+2};p)$. In the case $p=0$, the explicit formulas for $g(t_{n},t_{n+1}, t_{n+2};0) = g(t_{n},t_{n+1},t_{n+2})$ is given in proposition 6 in \cite{Roble_Rosales} as follows:
	%	\begin{equation}
		%		g(t_{n},t_{n+1},t_{n+2}) = \left\lfloor\frac{n}{2}\right\rfloor (t_n + t_{n+1} + t_{n+2} -1) -1
		%	\end{equation}
	%	or
	%	\begin{equation}
		%		g(t_{n},t_{n+1},t_{n+2})
		%		=
		%		\begin{cases}
			%			\frac{(3n-3)(n+1)(n+2)}{4} -1	\hspace{5pt}\text{if } n \text{ is odd};	\\
			%			\frac{(3n)(n+1)(n+2)}{4} -1	\hspace{15pt}\text{if } n \text{ is even.}
			%		\end{cases}
		%	\end{equation}
	%	

\vspace{0.5cm}

{\bf Acknowledgement:} This project was supported by the Development and Promotion for Science and Technology Talents Project (DPST), Thailand. I extend my sincere appreciation to Prof. Henrik Bachmann and Prof. Kohji Matsumoto for their invaluable guidance and support as my supervisors. I would like to thank the referee for helpful comments and corrections. 

\section{Preliminary Lemmas}
	
	Before proving Theorem \ref{Main-Thm}, we introduce some Lemmas. Beck and Kifer \cite{Beck_Kifer} show the following result on $g(a_{1}, a_{2}, \ldots, a_{k};s)$ in terms of $\ell = \gcd(a_{2}, a_{3}, \ldots, a_{k}) $.
	\begin{lemma}\label{lemma-from-Beck-Kifer}
		\cite[Lemma 4]{Beck_Kifer}
		For $k\geq 2$, let $A =(a_{1}, \ldots, a_{k})$ be a $k$-tuple of positive integers with $\gcd(A) = 1$. If $\ell = \gcd(a_{2}, a_{3}, \ldots, a_{k})$, let $a_{j} = \ell a_{j}'$ for $2 \leq j \leq k$. Then for $s\geq 0$
		\begin{align*}
			g(a_{1}, a_{2}, \ldots, a_{k}; s) &= \ell \, g\big(a_{1}, a_{2}', a_{3}', \ldots, a_{k}'; s\big) + a_{1}(\ell -1).
			%			\\
			%			n(a_{1}, a_{2}, \ldots, a_{k}; s) &= \ell \, n(a_{1},a_{2}', a_{3}', \ldots, a_{k}'; s) + \frac{1}{2}(a_{1}-1)(\ell-1), 
		\end{align*}
	\end{lemma}

	%	For $s\geq 1$, some explicit forms have been shown even such particular cases such as Triangular numbers \cite{Komatsu-triangular}, ...
	%	However, for $s\geq 1$, no explicit form has been shown even such particular cases because it is not easy to find $m_{i}^{(s)}$ in order to use Lemma \ref{lemma4}.

	The next lemma give an upper bound for the number of representations to $a_{1}x_{1} + \cdots + a_{k}x_{k} = g(a_{1}, \ldots, a_{k};s) - jc$, for all integers $j$ such that $0 \leq jc \leq g(a_{1}, \ldots, a_{k};s)$ when $c \equiv 0\pmod{a_{r}}$ for some  $r \in \{1,\ldots,k\}$.	
	\begin{lemma}\label{Lemma: d(g-jc) <= s}
		For $k\geq 2$, let $A = (a_{1}, \ldots, a_{k})$ be a $k$-tuple of positive integers with $\gcd(A) = 1$ and let $s \in \mathbb{Z}_{\geq 0}$. If $c$ is a positive integer such that $c \equiv 0 \pmod {a_{r}}$ for some $r = 1,\ldots,k$, then, for all integers $0 \leq jc \leq g(A;s)$, 
		$$
		d\big( g( A ;s)-jc ; A\big) \leq s.
		$$ 
	\end{lemma}
	\begin{proof}
		Suppose that $c\in\mathbb{Z}_{>0}$ such that $c \equiv 0\pmod {a_{r}}$ for some $r =1,\ldots,k$. 
		Assume that there exists $0 \leq j \leq \frac{g(A;s)}{c}$ such that
		\begin{equation*}
			d\big(g(A;s)-jc; A\big)
			\geq s+1. 
		\end{equation*}
		So, there are \emph{at least} $s+1$ non-negative integer solutions $(x_{1}, \ldots, x_{k})$ such that
		\begin{equation*}
			g(A;s)-jc
			=
			\sum_{\ell=1}^{k}x_{\ell}a_{\ell}.
		\end{equation*}
		Since $c \equiv 0 \pmod {a_{r}}$, then $ c =a_{r}q$ for some $q\in\mathbb{Z}_{\geq 0}$. So, we obtain that
		\begin{equation*}
			g(A;s)
			=
			x_{1}a_{1} + \cdots + x_{r-1}a_{r-1} + (x_{r}+jq)a_{r} + x_{r+1}a_{r+1} + \cdots + x_{k}a_{k},
		\end{equation*}
		this means that $g(A;s)$ has \emph{at least} $s+1$ non-negative representations in terms of $a_{1}, \ldots, a_{k}$. We get a contradiction since $g(A;s)$ must have at most $s$ representations. 
	\end{proof}
	
	%	In Lemma \ref{Lemma: d(g-jc) <= s}, the condition $c \equiv 0\pmod{a}$ or $c \equiv 0 \pmod{b}$ is needed. 
	%	\begin{example}
		%		Let $a =4, b=5, c=3,$ and $s =1$. Then $g(a,b;s)=g(4,5;1) = 40-4-5 = 31$. However, $d\big(g(4,5;1) - 3; 4,5\big) = d(28;4,5) = 2 > 1 = s.$
		%	\end{example}
	To accomplish the proof of Theorem \ref{Main-Thm}, we need the following lemma. 
	%Followed by Lemma \ref{Lemma: d(g-jc) <= s}, 
	If $k = 2$, says $A = (a,b)$, then, for a non-negative number $j \leq g(a,b;s)/c$, $d\big(g(a,b;s)-jc; a,b\big) = i$ is equivalent to $g(a,b;i-1) < g(a,b;s) - jc \leq g(a,b;i)$.
	
	\begin{lemma}\label{lemma: d(g_s-jc) = i iff g_i-1 < g_s-jc <= g_i}
		Let $a,b\in\mathbb{Z}_{>0}$ with $\gcd(a,b) = 1$, $s \in \mathbb{Z}_{\geq0}$, $i \in\{0,1,\ldots,s\}$. Suppose that $c$ is a positive integer such that $c  \equiv 0 \pmod a$ or $c \equiv 0 \pmod b$ and $j$ is a non-negative integer such that $j \leq \frac{g(a,b;s)}{c}$. Then
		\begin{equation*}
			d\big( g( a, b ;s)-jc ; a,b\big) = i,
		\end{equation*}
		if and only if,
		\begin{equation*}			
			g(a, b ;i-1) < g( a, b ;s)-jc \leq g(a, b ;i).
		\end{equation*}
		Here we set $g(a, b ;-1)$ to be $-2$.
	\end{lemma}
	\begin{proof}
		Without loss of generality, we only prove the case $c \equiv 0 \pmod{a}$. For convenient, throughout the proof, for $s \geq 0$, we let $g_{s} := g(a,b;s)$, which by \eqref{eq:twovariableclassical} is $g_{s} = (s+1)ab-a-b$.
		
		$(\Rightarrow)$ Suppose that $d(g_{s}-jc; a, b) = i$ for some $i = 0,1,\ldots,s$.
		By the definition of $g_{i}$, it follows immediately that 
		$g_{s}-jc \leq g_i$. Clearly, if $i=0$, then $-2 < g_{s}-jc \leq g_{i}$, we are done. So, assume that $i\geq1$. If $g_s-jc = g_{i-1}$, then $d(g_{s}-jc; a, b) = i-1$, a contradiction. It remains to show that 
		\begin{equation*}
			g_{i-1}  < g_{s}-jc.
		\end{equation*}
		We will prove this statement by assuming that $g_{s} -jc< g_{i-1}$. 
		Since $d(g_{s}-jc; a, b) = i$, by the definition of $d$, there are $i$ non-negative integer solutions $(x,y)$ such that
		\begin{equation*}
			g_{s}-jc = xa + yb
		\end{equation*}
		If we let $\Delta = g_{i-1} - g_{s} + jc $, then $\Delta > 0$ and 
		\begin{equation*}
			\Delta = 
			jc 
			- (s+1-i)(ab).
		\end{equation*}
		Since $c = ak$ for some $k\in\mathbb{Z}_{>0}$, it follows that
		\begin{equation*}
			\Delta
			=
			\big(jk-(s+1-i)b\big)a.
		\end{equation*}
		Therefore, we obtain
		\begin{equation*}
			g_{i-1} = g_{s}-jc + \Delta = \big(x+jk-(s+1-i)b\big)a + yb,
		\end{equation*}
		this implies that $g_{i-1} = g(a,b; i-1)$ has at least $i$ representations in terms of $a$ and $b$, a contradiction. Therefore, $g_{i-1} < g_{s}-jc \leq g_{i}$.
		
		$(\Leftarrow)$ Suppose that $g_{i-1} < g_{s}-jc \leq g_{i}$ for some $i=0,1,2,\ldots,s$. Clearly, since $g_{i-1} < g_{s}-jc$, we have
		\begin{equation*}
			d(g_{s}-jc; a, b) \geq i.
		\end{equation*}
		Our goal is to show that $d(g_{s}-jc; a, b) = i$. If $g_{s}-jc = g_{i}$, then we are done. So we assume that $g_{s}-jc < g_{i}$ and also assume that $d(g_{s}-jc; a, b) > i$. Then, there are \emph{at least} $i+1$ non-negative integer solutions $(x,y)$ such that
		\begin{equation*}
			g_{s}-jc = xa + yb.
		\end{equation*}
		Let $\Delta = g_{i} - g_{s}+jc$. Then $\Delta > 0$ and 
		\begin{equation*}
			\Delta
			=
			jc - (s-i)ab.
		\end{equation*}
		We have $c = ka$ for some $k\in\mathbb{Z}_{>0}$. Thus,
		\begin{equation*}
			\Delta = \big(jk - (s-i)b\big)a.
		\end{equation*}
	Therefore, we obtain that
		\begin{equation*}
			g_{i} = g_{s}-jc + \Delta = \big(x+jk - (s-i)b\big)a + yb,
		\end{equation*}
		i.e., $g_{i}$ has \emph{at least} $i+1$ representations in terms of $a$ and $b$, a contradiction. Therefore,
		\begin{equation*}
			d(g_{s}-jc; a, b) = i.
		\end{equation*}
		Similarly to that, we can prove the case where $c \equiv 0 \pmod{b}$.
	\end{proof}	
    %==================== new lemma
     \begin{lemma} \label{lemma for optimal m}
		Let $a,b \in \mathbb{Z}_{>0}$ with $a<b$, $\gcd(a,b) = 1$, and let $s, K \in \mathbb{Z}_{\geq 0}$. If m is an integer such that $m > g(a,b;s) + Ka$, then, for all $j \in \mathbb{Z}_{\geq 0}$, we have
		\begin{equation*}
			d\big( m - ja; a,b\big) \geq d\big(g(a,b;s) + (K -j)a; a,b\big).
		\end{equation*}
	\end{lemma}
	\begin{proof}
		We will prove by induction on $j$. If $j = 0$, we can assume that there exists a non-negative integer $\ell$ such that $d\big(g(a,b;s) + Ka; a,b\big) = \ell$. By Lemma \ref{lemma: d(g_s-jc) = i iff g_i-1 < g_s-jc <= g_i}, we have 
		\begin{equation*}
			g(a,b; \ell-1) < g(a,b;s) + Ka \leq g(a,b; \ell). 
		\end{equation*}
		Hence $m > g(a,b;\ell-1)$, which means
		\begin{equation*}
			d(m ; a,b) \geq \ell = d\big(g(a,b;s) + Ka; a,b\big).
		\end{equation*}
		The base step is proved. So let $j$ be a non-negative integer and assume that
		\begin{equation*}
			d\big( m - ja; a,b\big) \geq d\big(g(a,b;s) + (K -j)a; a,b\big).
		\end{equation*}
		We want to show that
		\begin{equation*}
			d\big( m - (j+1)a; a,b\big) \geq d\big(g(a,b;s) + (K -(j+1))a; a,b\big).
		\end{equation*}
		Assume that $d\big( m - (j+1)a; a,b\big) < d\big(g(a,b;s) + (K -(j+1))a; a,b\big)$.
		Suppose that there exists $\ell \geq 0$ such that $d\big( m - (j+1)a; a,b\big) = \ell$. So
		\begin{equation*}
			\ell < d\big(g(a,b;s) + (K -(j+1))a; a,b\big).
		\end{equation*}
		Observe that $d\big(g(a,b;s) + (K -(j+1))a; a,b\big) \leq d\big(g(a,b;s)+(K-j)a; a,b\big)$ by Lemma \ref{lemma: d(g_s-jc) = i iff g_i-1 < g_s-jc <= g_i}.
		One can see that since $d\big( m - (j+1)a; a,b\big) = \ell$, then
		\begin{equation*}
			d\big( m - ja; a,b\big)
			=
			\begin{cases}
				\ell + 1 	&\text{ if } b \mid (m - ja),
				\\
				\ell		&\text{ otherwise.}
			\end{cases}
		\end{equation*}
		If $ d\big( m - ja; a,b\big) = \ell$, then we get a contradiction that 
		\begin{equation*}
			\ell = d\big( m - ja; a,b\big) \geq d\big(g(a,b;s)+ (K-j)a; a,b\big) \geq  d\big(g(a,b;s) + (K-(j+1))a; a,b\big) > \ell.
		\end{equation*}
		If $d\big( m - ja; a,b\big) = \ell+1$, then 
		\begin{equation*}
			\ell < d\big(g(a,b;s) + (K -(j+1))a; a,b\big) \leq d\big(g(a,b;s)+(K-j)a; a,b\big) \leq d\big( m - ja; a,b\big) = \ell+1.
		\end{equation*}
		It follows that
		\begin{equation*}
			d\big(g(a,b;s) + (K -(j+1))a; a,b\big) = d\big(g(a,b;s)+(K-j)a; a,b\big) = d\big( m - ja; a,b\big) = \ell +1.
		\end{equation*}
		Since $d\big(g(a,b;s) + (K -(j+1))a; a,b\big) = \ell +1$, by Lemma \ref{lemma: d(g_s-jc) = i iff g_i-1 < g_s-jc <= g_i}, $g(a,b;s) + (K -(j+1))a > g(a,b;\ell)$. This means $ m - (j+1)a > g(a,b;s) + (K -(j+1))a > g(a,b;\ell)$, thus
		\begin{equation*}
			d\big( m - (j+1)a; a,b\big) \geq \ell +1,
		\end{equation*}
		which contradicts with $d(m-(j+1)a; a,b) = \ell$. Therefore, 
        \begin{equation*}
        d\big( m - (j+1)a; a,b\big) \geq d\big(g(a,b;s) + (K - (j+1))a; a,b\big).
        \tag*{\qedhere}
        \end{equation*}
	\end{proof}
	%======================
	\section{Proof of Theorem \ref{Main-Thm}}\label{Section-ProofThm1}
	By applying Lemma \ref{lemma-from-Beck-Kifer}, Lemma \ref{Lemma: d(g-jc) <= s} and Lemma \ref{lemma: d(g_s-jc) = i iff g_i-1 < g_s-jc <= g_i}, we can prove Theorem \ref{Main-Thm} as follows.
	\begin{proof}[Proof of Theorem \ref{Main-Thm}]
		Suppose that $d_{1} = \gcd(a_{2},a_{3})$ and $a_{1} \equiv 0\pmod{\frac{a_{2}}{d_{1}}}$.
		By applying Lemma \ref{lemma-from-Beck-Kifer}, we obtain that
		\begin{equation}\label{eq:main point in proof THM}
			g\bigg( 
			a_{1}, a_{2},a_{3}; \sum_{j=0}^{t}\left\lceil\frac{j a_{2}a_{3}}{a_{1}d_{1}^{2}} \right\rceil \bigg)
			=
			d_{1} g\bigg(a_{1}, \frac{a_{2}}{d_{1}}, \frac{a_{3}}{d_{1}}; \sum_{j=0}^{t}\left\lceil\frac{j a_{2}a_{3}}{a_{1}d_{1}^{2}}\right\rceil\bigg)
			+
			a_{1}(d_{1}-1).
		\end{equation}
		We will show that
		\begin{equation}\label{eq-relation}
			g\bigg(a_{1}, \frac{a_{2}}{d_{1}}, \frac{a_{3}}{d_{1}}; \sum_{j=0}^{t}\left\lceil\frac{j a_{2}a_{3}}{a_{1}d_{1}^{2}}\right\rceil\bigg)
			= 
			g\bigg( \frac{a_{2}}{d_{1}}, \frac{a_{3}}{d_{1}} ;t\bigg). 
		\end{equation}
		Then one can see that, for $m \in\mathbb{Z}_{\geq 0}$,
		\begin{align}\label{eq:the numbe d(m) = sum d(m-ja)}
			d \bigg(m; a_{1}, \frac{a_{2}}{d_{1}}, \frac{a_{3}}{d_{1}}\bigg) 
			&= \sum_{j=0}^{\lfloor\frac{m}{a_{1}}\rfloor} d\bigg(m-ja_{1} ; \frac{a_{2}}{d_{1}}, \frac{a_{3}}{d_{1}}\bigg).
		\end{align}
		Put $m = g\bigg( \frac{a_{2}}{d_{1}}, \frac{a_{3}}{d_{1}} ;t\bigg)$, then we obtain
		\begin{align}\label{eq:put m=g_s}
			d \Bigg(g\bigg( \frac{a_{2}}{d_{1}}, \frac{a_{3}}{d_{1}} ;t\bigg); a_{1}, \frac{a_{2}}{d_{1}}, \frac{a_{3}}{d_{1}}\Bigg) 
			= \sum_{j=0}^{\left\lfloor\frac{g\big( \frac{a_{2}}{d_{1}}, \frac{a_{3}}{d_{1}} ;t\big)}{a_{1}}\right\rfloor} d\Bigg(g\bigg( \frac{a_{2}}{d_{1}}, \frac{a_{3}}{d_{1}} ;t\bigg)-ja_{1} ; \frac{a_{2}}{d_{1}}, \frac{a_{3}}{d_{1}}\Bigg).
		\end{align}
		
		By Lemma  \mbox{\ref{Lemma: d(g-jc) <= s}\strut}, we have that each value of $d\bigg(g\big( \frac{a_{2}}{d_{1}}, \frac{a_{3}}{d_{1}} ;t\big)-ja_{1} ; \frac{a_{2}}{d_{1}}, \frac{a_{3}}{d_{1}}\bigg)$ have to be equal to any of $0,1,\ldots, t$. To calculate the right-hand side of \eqref{eq:put m=g_s}, we count the number of $ 0 \leq j \leq g\big( \frac{a_{2}}{d_{1}}, \frac{a_{3}}{d_{1}} ;t\big) / a_{1}$ such that
		\begin{equation}\label{eq:di}
			d\Bigg(g\bigg( \frac{a_{2}}{d_{1}}, \frac{a_{3}}{d_{1}} ;t\bigg)-ja_{1} ; \frac{a_{2}}{d_{1}}, \frac{a_{3}}{d_{1}}\Bigg) = i,
		\end{equation}
		for all $i = 1,2,\ldots, t$. For convenient, for an integer $i$, let $g_{i} := g\big( \frac{a_{2}}{d_{1}}, \frac{a_{3}}{d_{1}} ;i\big)$. For given $i$ the $j$ such that \eqref{eq:di} holds are, by Lemma \ref{lemma: d(g_s-jc) = i iff g_i-1 < g_s-jc <= g_i}, those with $g_{i-1} < g_{t} - ja_{1} \leq  g_{i}$. By \eqref{g(a,b;p)}, this is equivalent to
		\begin{align*}
			\frac{ia_{2}a_{3}}{d_{1}^{2}} - \frac{a_{2}}{d_{1}} - \frac{a_{3}}{d_{1}}
			\,<\,
			(t+1)\frac{a_{2}a_{3}}{d_{1}^{2}} - \frac{a_{2}}{d_{1}} - \frac{a_{3}}{d_{1}} - ja_{1}
			\,\leq\,
			(i+1)\frac{a_{2}a_{3}}{d_{1}^{2}} - \frac{a_{2}}{d_{1}} - \frac{a_{3}}{d_{1}}.
		\end{align*}
		So,
		\begin{equation*}
			(t+1-i)\frac{a_{2}a_{3}}{a_{1}d_{1}^{2}}
			> j
			\geq
			(t-i)\frac{a_{2}a_{3}}{a_{1}d_{1}^{2}}.
		\end{equation*}
		Thus, by Lemma \ref{lemma: d(g_s-jc) = i iff g_i-1 < g_s-jc <= g_i}, there are 
		\begin{equation*}
			\left\lceil (t+1-i)\frac{a_{2}a_{3}}{a_1d_{1}^{2}} \right\rceil -
			\left\lceil (t-i)\frac{a_{2}a_{3}}{a_1d_{1}^{2}} \right\rceil 
		\end{equation*}
		of $j$ in $[0, g_{t}/a_{1})$ such that $d\big(g_{t}-ja_{1}; \frac{a_{2}}{d_{1}}, \frac{a_{3}}{d_{1}}\big) = i $ for $i = 1,2,\ldots,t$. So, by \eqref{eq:put m=g_s}, we have 
		\begin{align*}
			&d \Bigg(g\Big(\frac{a_{2}}{d_{1}}, \frac{a_{3}}{d_{1}} ;t\Big); a_{1}, \frac{a_{2}}{d_{1}}, \frac{a_{3}}{d_{1}}\Bigg) 
			\nonumber
			\\
			&= \sum_{j=0}^{\big\lfloor\frac{g_{t}}{a_{1}}\big\rfloor} d\Bigg(g\bigg(\frac{a_{2}}{d_{1}}, \frac{a_{3}}{d_{1}} ;t\bigg)-ja_{1} ; \frac{a_{2}}{d_{1}}, \frac{a_{3}}{d_{1}} \Bigg)
			\\
			&= 
			t
			\left\lceil \frac{a_{2}a_{3}}{a_{1}d_{1}^{2}} \right\rceil
			+
			(t-1)
			\bigg(
			\left\lceil \frac{2a_{2}a_{3}}{a_{1}d_{1}^{2}} \right\rceil -
			\left\lceil \frac{a_{2}a_{3}}{a_{1}d_{1}^{2}} \right\rceil 
			\bigg)
			+
			(t-2)
			\bigg(
			\left\lceil \frac{3a_{2}a_{3}}{a_{1}d_{1}^{2}} \right\rceil - \left\lceil \frac{2a_{2}a_{3}}{a_{1}d_{1}^{2}} \right\rceil \bigg) +
			\\
			&\hspace{5mm}
			\cdots
			+
			\bigg(
			\left\lceil \frac{ta_{2}a_{3}}{a_{1}d_{1}^{2}} \right\rceil -
			\left\lceil \frac{(t-1)a_{2}a_{3}}{a_{1}d_{1}^{2}} \right\rceil 
			\bigg)
			\\
			&=
			\sum_{j=0}^{t}
			\left\lceil \frac{ja_{2}a_{3}}{a_{1}d_{1}^{2}} \right\rceil.
		\end{align*}
		Therefore, 
		\begin{equation*}
			g\Big(\frac{a_{2}}{d_{1}}, \frac{a_{3}}{d_{1}} ; t\Big)
			\leq 
			g\Bigg(a_{1}, \frac{a_{2}}{d_{1}}, \frac{a_{3}}{d_{1}}; 
			\sum_{j=0}^{t}
			\left\lceil \frac{j a_{2}a_{3}}{a_{1}d_{1}^{2}} \right\rceil  \Bigg).
		\end{equation*}
        If we let $m > g\Big(\frac{a_{2}}{d_{1}}, \frac{a_{3}}{d_{1}} ; t\Big)$, then, by the definition of $g\Big(\frac{a_{2}}{d_{1}}, \frac{a_{3}}{d_{1}} ; t\Big)$,
        \begin{equation*}
        d\Big( m; \frac{a_2}{d_1}, \frac{a_3}{d_1} \Big) 
        >
        d\Big( g\Big(\frac{a_{2}}{d_{1}}, \frac{a_{3}}{d_{1}} ; t\Big); \frac{a_2}{d_1}, \frac{a_3}{d_1} \Big)
        \end{equation*}
        Since $a_1$ is divisible by $a_2/d_1$, followed by Lemma \ref{lemma for optimal m} when $K=0$, we have
        \begin{align*}
            d\Big(  m; &a_1, \frac{a_2}{d_1}, \frac{a_3}{d_1}    \Big)
            =
            d\Big(m; \frac{a_2}{d_1}, \frac{a_3}{d_1}\Big) 
                + \sum_{j=1}^{\lfloor m/a_1 \rfloor} d\Big(m-ja_{1} ; \frac{a_{2}}{d_{1}}, \frac{a_{3}}{d_{1}} \Big)
                \\
            &> d\Big( g\Big(\frac{a_{2}}{d_{1}}, \frac{a_{3}}{d_{1}} ; t\Big); \frac{a_2}{d_1}, \frac{a_3}{d_1} \Big) + \sum_{j=1}^{\lfloor m/a_1\rfloor} d\Big(m-ja_{1} ; \frac{a_{2}}{d_{1}}, \frac{a_{3}}{d_{1}} \Big)
            \\
            &\geq d\Big( g\Big(\frac{a_{2}}{d_{1}}, \frac{a_{3}}{d_{1}} ; t\Big); \frac{a_2}{d_1}, \frac{a_3}{d_1} \Big)
             +
             \sum_{j=1}^{\lfloor g_t/a_1\rfloor} d\Bigg(g\bigg(\frac{a_{2}}{d_{1}}, \frac{a_{3}}{d_{1}} ;t\bigg)-ja_{1} ; \frac{a_{2}}{d_{1}}, \frac{a_{3}}{d_{1}} \Bigg) 
             =
			\sum_{j=0}^{t}
			\left\lceil \frac{ja_{2}a_{3}}{a_{1}d_{1}^{2}} \right\rceil.
        \end{align*}
        Therefore the value $g\Big(\frac{a_{2}}{d_{1}}, \frac{a_{3}}{d_{1}} ; t\Big)$ is the largest one having $\sum_{j=0}^{t}
		\left\lceil \frac{j a_{2}a_{3}}{a_{1}d_{1}^{2}} \right\rceil$ representations in terms of $ a_{1}, \frac{a_{2}}{d_{1}}$ and $ \frac{a_{3}}{d_{1}} $. 
		Then
		\begin{equation*}
			g\Big(\frac{a_{2}}{d_{1}}, \frac{a_{3}}{d_{1}} ; t\Big)
			=
			g\Bigg(a_{1}, \frac{a_{2}}{d_{1}}, \frac{a_{3}}{d_{1}}; 
			\sum_{j=0}^{t}
			\left\lceil \frac{j a_{2}a_{3}}{a_{1}d_{1}^{2}} \right\rceil  \Bigg).
		\end{equation*}
		Hence, by \eqref{eq:main point in proof THM} and \eqref{g(a,b;p)},
		\begin{align*}
			g\Bigg( 
			a_{1}, a_{2}, a_{3}; \sum_{j=0}^{t}\left\lceil\frac{j a_{2}a_{3}}{a_{1}d_{1}^{2}} \right\rceil \Bigg)
			&=
			d_{1} g\Big(\frac{a_{2}}{d_{1}}, \frac{a_{3}}{d_{1}} ; t\Big)
			+
			a_{1}(d_{1}-1)
			\\
			&= d_{1}\Big((t+1)\frac{a_{2}a_{3}}{d_{1}^{2}}-\frac{a_{2}}{d_{1}}- \frac{a_{3}}{d_{1}}\Big) +a_{1}d_{1} -a_{1}
			\\
			&= (t+1)\frac{a_{2}a_{3}}{d_{1}} + a_{1}d_{1} - a_{1} - a_{2} - a_{3}.
            \tag*{\qedhere}
		\end{align*}
        \end{proof}

	Compared to the results in \cite{Komatsu-The Frobenius number repunits, Komatsu-Ying-Frob_Fibonacci} our main theorem seems more useful when $t$ is large, since their results have an upper bound on $t$. The result in \cite{Binner-binner2021bounds} holds for $s$ is extremely large. For example, by \cite[Section 3.2]{Binner-binner2021bounds}), $g(16,23,37;s)$ can be found for $s \geq 157291918$. Therefore, our result behaves nicely for $s$ not too large. In \cite{Woods-woods2022generalized} the value for $s$ is not explicitly given. 
	%	\begin{remark}
		%		%\label{Remark after Thm1}
		%		In the statement of Theorem \ref{Main-Thm}, we have that $\gcd(q_{1},q_{3}) = \gcd(q_{2},q_{3}) = \gcd(q_{2},q_{4}) = 1$. The condition $ \gcd(q_{1}q_{2},q_{3}q_{4}) =1$ is not required. So, $\gcd(q_{1},q_{4})$ can be greater than or equal to 1. (How are about $\gcd(q_{1},q_{2})$ and $\gcd(q_{3},q_{4})$?)
		%	\end{remark}
	
	\section{Further Properties and Special Cases}\label{section: special cases}
	%\section{Special cases of Theorem \ref{Main-Thm}}
	
	The next proposition shows that, for given $a ,b \in\mathbb{Z}_{> 0}$ with $\gcd(a,b)=1$, if $c \in \mathbb{Z}_{>0}$ such that $c\equiv 0\pmod{a}$ or $c\equiv 0\pmod{b}$, then the sequence $\big( d\big(g(a, b; s)-jc; a,b\big)\big)_{j\geq0}$ is decreasing.
	\begin{prop}%\label{prop: the number of rep. is decresing}
		Let $a,b\in\mathbb{Z}_{>0}$ with $\gcd(a,b) = 1$ and let $s \in \mathbb{Z}_{\geq0}$. Suppose that $c \in \mathbb{Z}_{>0}$ such that $c  \equiv 0 \pmod a$ or $c \equiv 0 \pmod b$. If $j_{1}, j_{2} \in\mathbb{Z}_{\geq0}$ such that $0 \leq j_{1} < j_{2} \leq \frac{g(a,b;s)}{c}$, then
		\begin{equation*}
			d \big(g(a, b; s)
			-j_{2}c	;a, b\big)
			\leq d \big(	g(a, b; s)-	j_{1}c;	a, b\big).
		\end{equation*}
	\end{prop}
	\begin{proof}
		For convenient, we let $g_{s} := g\big(a, b; s\big)$. 
		According to Lemma \ref{Lemma: d(g-jc) <= s}, we can assume that 
		\begin{equation*}
			d(g_{s} -	j_{1}c;	a, b) = i,
		\end{equation*}
		for some $i =0,1,\ldots,s$. 
		If $i =0$, then, by Lemma \ref{lemma: d(g_s-jc) = i iff g_i-1 < g_s-jc <= g_i},
		\begin{equation*}
			0 \leq  g_{s}-j_{1}c \leq g_{0}.
		\end{equation*}
		Since $j_{1}< j_{2} \leq \dfrac{g_{s}}{c}$, we have
		\begin{equation*}
			0 \leq g_{s}- j_{2}c < g_{s}-j_{1}c \leq g_{0}.
		\end{equation*}
		Then, by Lemma \ref{lemma: d(g_s-jc) = i iff g_i-1 < g_s-jc <= g_i}, 
		\begin{equation*}
			d \big(g_{s} -	j_{2}c;	a, b\big)=0
			=
			d \big(	g_{s} -	j_{1}c	;	a, b\big).
		\end{equation*} 
		Assume that $i \geq 1$. Again by Lemma \ref{lemma: d(g_s-jc) = i iff g_i-1 < g_s-jc <= g_i} , we have
		\begin{equation*}
			g_{i-1} < g_{s} - j_{1}c \leq g_{i}.
		\end{equation*}
		If $g_{i-1} < g_{s} - j_{2}c < g_{s} - j_{1}c \leq g_{i}$, It follows immediately that
		\begin{equation*}
			d\big(g_{s}-j_{2}c; a, b\big) = i = d\big(g_{s}-j_{1}c; a, b\big).
		\end{equation*}
		If $g_{s} - j_{2}c \leq g_{i-1}$, then, without loss of generality, assume that there exists $ k \in \{1,2,\ldots,i-1 \} $ such that
		\begin{equation*}
			g_{k-1} < g_{s} - j_{2}c \leq g_{k}.
		\end{equation*}
		Therefore, by Lemma \ref{lemma: d(g_s-jc) = i iff g_i-1 < g_s-jc <= g_i}, 
		\begin{equation*}
			d \big(
			g_{s} -	j_{2}c
			;
			a, b
			\big)
			=
			k
			< i 
			=
			d \big(
			g_{s} -	j_{1}c
			;
			a, b
			\big).\tag*{\qedhere}
		\end{equation*}
	\end{proof}
	
	For example, if we let $a =3, b=7, c=6,$ and $s =2$. Then $g(a,b;s)=g(3,7;2) = 63-3-7 = 53$. We have $6 \equiv 0\pmod{3}$. Then, 
		\begin{align*}
			\bigg(d\big(g(3,7;2) - 6j; 3,7\big)\bigg)_{j\geq0} &= \big(d(63-6j;4,5)\big)_{j \geq 0}
			%\\
			%&= \big(d(63;3,7), d(57;3,7), d(51;3,7), d(45;3,7), d(39;3,7), \ldots\big)
			\\
			&= (4,3,3,3,2,2,2,2,1,1,1,0,0,\ldots),
		\end{align*}
		which is decreasing.
	However, if $a =4, b=5, c=3,$ and $s =1$. Then $g(a,b;s)=g(4,5;1) = 40-4-5 = 31$. We have $c \not\equiv 0\pmod{a}$ and $c \not\equiv 0\pmod{b}$ and the sequence
		\begin{align*}
			\bigg(d\big(g(4,5;1) - 3j; 4,5\big)\bigg)_{j\geq0} &= \big(d(31-3j;4,5)\big)_{j \geq 0}
			%\\
			%&= \big(d(31;4,5), d(28;4,5), d(25;4,5), d(22;4,5), d(19;4,5), \ldots\big)
			\\
			&= (1,2,2,1,1,1,1,1,0,1,0,0,\ldots),
		\end{align*}
		which is not decreasing.
	%Recall that, in Theorem \ref{Main-Thm}, the necessary conditions for positive integers $q_{1}, q_{2}, q_{3},$ and $q_{4}$ are $q_{1} < q_{2}, q_{2}<q_{4},$ and $\gcd(q_{1},q_{3}) = 1 = \gcd(q_{2},q_{4}) = \gcd(q_{2},q_{3})$. However, for $\gcd(q_{1},q_{2}), \gcd(q_{1},q_{4}),$ and $\gcd(q_{3},q_{4})$, they can be greater than $1$. But we must be careful to consider these conditions that do not contradict with the largest common divisor $\gcd(q_{1}q_{2}, q_{2}q_{3}, q_{3}q_{4}) = 1$.
	
	The first application of Theorem 1 is to calculate the generalized Frobenius number for three consecutive triangular integers, which are the numbers of dots in an equilateral triangle. The explicit formula for the $n$th triangular number is given by $t_{n} = \binom{n+1}{2} = \frac{n(n+1)}{2}$. Robles-P\'{e}rez and Rosales \cite{Roble_Rosales} show that, for $n\in\mathbb{Z}_{>0}$, $\gcd(t_{n}, t_{n+1}, t_{n+2}) =1$ and 
	\begin{equation*}
		\gcd(t_{n+1}, t_{n+2}) 
		=
		\begin{cases}
			\frac{n+2}{2} &\text{ if } n \text{ is even;}
			\\
			n+2		&\text{ if } n \text{ is odd.}
		\end{cases}
	\end{equation*}
	
	\begin{corollary}\label{cor-Komatsu}
		\cite[Theorem 1]{Komatsu-triangular}
		Let $n \in \mathbb{Z}_{> 0}$ and $s \in \mathbb{Z}_{\geq 0}$. If  $d_{1} = \gcd(t_{n+1}, t_{n+2})$ and $d_{3} = \gcd(t_{n},t_{n+1})$, then, we have 
		\begin{equation}\label{eq: 1st triangular formula}
			g\Big( 
			t_{n}, t_{n+1}, t_{n+2}; \sum_{j=0}^{t}\left\lceil\frac{j t_{n+1}t_{n+2}}{t_{n}d_{1}^{2}} \right\rceil \Big)
			=(t+1)\frac{t_{n+1}t_{n+2}}{d_{1}} + t_{n}d_{1} -t_{n}-t_{n+1}-t_{n+2},
		\end{equation}
		and
		\begin{equation}\label{eq: 2nd triangular formula}
			g\Big( 
			t_{n}, t_{n+1}, t_{n+2}; \sum_{j=0}^{t}\left\lceil\frac{j t_{n}t_{n+1}}{t_{n+2}d_{3}^{2}} \right\rceil \Big)
			=(t+1)\frac{t_{n}t_{n+1}}{d_{3}} + t_{n+2}d_{3} -t_{n}-t_{n+1}-t_{n+2}.
		\end{equation}
		%If $d_{2} = \gcd(t_{n},t_{n+2})$
	\end{corollary}
	\begin{proof}
        We provide the proof for the first case, as the proofs of the other cases follow similarly with analogous arguments. For completeness, we outline the first case below:
        
		If $n$ is even, then $d_{1} = \gcd(t_{n+1}, t_{n+2}) = \frac{(n+2)}{2}$ and $\frac{t_{n+1}}{d_{1}} = n+1$. Since $t_{n} = \frac{n(n+1)}{2} \equiv 0 \pmod{n+1}$, by applying Theorem \ref{Main-Thm}, we have 
		\begin{equation*}
			g\Big(t_{n},t_{n+1},t_{n+2}; \sum_{j=0}^{t}\left\lceil\frac{j t_{n+1}t_{n+2}}{t_{n}d_{1}^{2}} \right\rceil \Big)
			=(t+1)\frac{t_{n+1}t_{n+2}}{d_{1}} + t_{n}d_{1} -t_{n}-t_{n+1}-t_{n+2}.
		\end{equation*}
		
		Since $n$ is even, it follows that $d_{3} = \gcd(t_{n},t_{n+1}) = n+1$ and $\frac{t_{n+1}}{d_{3}} = \frac{n+2}{2}$. Since $t_{n+2} = \frac{(n+2)(n+3)}{2} \equiv 0 \pmod{\frac{n+2}{2}}$, by Remark \ref{remark: after thm1}, we obtain
		\begin{equation*}
			g\Big( 
			t_{n}, t_{n+1}, t_{n+2}; \sum_{j=0}^{t}\left\lceil\frac{j t_{n}t_{n+1}}{t_{n+2}d_{3}^{2}} \right\rceil \Big)
			=(t+1)\frac{t_{n}t_{n+1}}{d_{3}} + t_{n+2}d_{3} -t_{n}-t_{n+1}-t_{n+2}.
		\tag*{\qedhere}
            \end{equation*}
		% If $n$ is odd, then $d_{1} = \gcd(t_{n+1}, t_{n+2}) = n+2$ and $\frac{t_{n+1}}{d_{1}} = \frac{n+1}{2}$. By applying Theorem \ref{Main-Thm} since $t_{n} = \frac{n(n+1)}{2} \equiv 0 \pmod{\frac{n+1}{2}}$, we have
		% \begin{equation*}
		% 	g\Big(t_{n},t_{n+1},t_{n+2}; \sum_{j=0}^{t}\left\lceil\frac{j t_{n+1}t_{n+2}}{t_{n}d_{1}^{2}} \right\rceil \Big)
		% 	=(t+1)\frac{t_{n+1}t_{n+2}}{d_{1}} + t_{n}d_{1} -t_{n}-t_{n+1}-t_{n+2}.
		% \end{equation*}
		
		% Since $n$ is odd, it follows that $d_{3} = \gcd(t_{n},t_{n+1}) = \frac{n+1}{2}$, $\frac{t_{n+1}}{d_{3}} = n+2$, and $t_{n+2} = \frac{(n+2)(n+3)}{2} \equiv 0 \pmod{n+2}$. Then, by Remark \ref{remark: after thm1}, we have
		% \begin{equation*}
		% 	g\Big( 
		% 	t_{n}, t_{n+1}, t_{n+2}; \sum_{j=0}^{t}\left\lceil\frac{j t_{n}t_{n+1}}{t_{n+2}d_{3}^{2}} \right\rceil \Big)
		% 	=(t+1)\frac{t_{n}t_{n+1}}{d_{3}} + t_{n+2}d_{3} -t_{n}-t_{n+1}-t_{n+2}.
		% \end{equation*}
	\end{proof}
	
	Note that we can write \eqref{eq: 1st triangular formula} in Corollary \ref{cor-Komatsu} as the same result in \cite{Komatsu-triangular}: For even $n$, we have
	\begin{equation*}
		g\Big(t_{n}, t_{n+1}, t_{n+2}; t(t+1) +\sum_{j=1}^{t}\left\lceil\frac{6j}{n}\right\rceil \Big) 
		=
		\frac{(n+1)(n+2)(2t(n+3) + 3n)}{4} -1,
	\end{equation*}
	and, for odd $n \geq 3$,
	\begin{equation*}
		g\Big(t_{n}, t_{n+1}, t_{n+2}; \sum_{j=1}^{t}\left\lceil \frac{j}{2} \big(1 + \frac{3}{n}\big) \right\rceil \Big) 
		=
		\frac{(n+1)(n+2)((n+3)t + 3(n-1))}{4} -1.
	\end{equation*}
	
	\begin{example}\label{ex: triangular numbers}
		Let $n=4$. Then $(t_{4},t_{5},t_{6}) = (10,15,21)$. We have $d_{1} = \gcd(15,21) = 3$. For convenient, we let $\Sigma_{1} := \sum_{j=0}^{t}\left\lceil\frac{j t_{5}t_{6}}{t_{4}d_{1}^{2}} \right\rceil$. Then, by \eqref{eq: 1st triangular formula}, for $t \geq 0$, $g\big( 
		t_{4}, t_{5}, t_{6}; \Sigma_{1} \big)$ are shown as follow:
		\begin{center}
			\begin{tabular}{|c|ccccccccccc|}
				\hline
				$t$ & 0 & 1 & 2 & 3 & 4 & 5 & $\ldots$ & 100 & $\ldots$ & $10^{4}$ & $\ldots$  \\
				\hline
				$\Sigma_{1}$ & 0 & 4 & 11 & 22 & 36 & 54 & $\ldots$ & 17700 & $\ldots$ & 175020000 & $\ldots$ \\
				\hline
				$g(10,15,21;\Sigma_{1})$ & 89 & 194 & 299 & 404 & 509 & 614 & $\ldots$ & 10589 & $\ldots$ & 1050089 & $\ldots$  \\
				\hline
			\end{tabular} 
		\end{center}
		In the same way, $d_{3} = \gcd(10,15) = 5$. If $\Sigma_{2} := \sum_{j=0}^{t}\left\lceil\frac{j t_{4}t_{5}}{t_{6}d_{3}^{2}} \right\rceil$. Then, by \eqref{eq: 2nd triangular formula}, for $t \geq 0$, $g\big( 
		t_{4}, t_{5}, t_{6}; \Sigma_{2} \big)$ are shown as follow:
		\begin{center}
			\begin{tabular}{|c|ccccccccccc|}
				\hline
				$t$ & 0 & 1 & 2 & 3 & 4 & 5 & $\ldots$ & 100 & $\ldots$ & $10^{4}$ & $\ldots$  \\
				\hline
				$\Sigma_{2}$ & 0 & 1 & 2 & 3 & 5 & 7 & $\ldots$ & 1486 & $\ldots$ & 14291429 & $\ldots$ \\
				\hline
				$g(10,15,21;\Sigma_{2})$ & 89 & 119 & 149 & 179 & 209 & 239 & $\ldots$ & 3089 & $\ldots$ & 300089 & $\ldots$  \\
				\hline
			\end{tabular} 
		\end{center}
	\end{example}
	
	%	\begin{example}
		%		Let $n=7$. Then $(t_{7},t_{8},t_{9}) = (28,36,45)$ and $\ell = 9$. For convenient, let $\Sigma := \sum_{j=0}^{s}\left\lceil\frac{j t_{8}t_{9}}{t_{7}\ell^{2}} \right\rceil$. Then, for $s \geq 0$, $g\big( 
		%		t_{7}, t_{8}, t_{9}; \Sigma \big)$ are shown as follow:
		%		\begin{center}
			%			\begin{tabular}{|c|ccccccccccc|}
				%				\hline
				%				s & 0 & 1 & 2 & 3 & 4 & 5 & $\ldots$ & 100 & $\ldots$ & $10^{4}$ & $\ldots$  \\
				%				\hline
				%				$\Sigma$ & 0 & 1 & 3 & 6 & 9 & 13 & $\ldots$ & 3650 & $\ldots$ & 35722143 & $\ldots$ \\
				%				\hline
				%				$g(28,36,45;\Sigma)$ & 89 & 194 & 299 & 404 & 509 & 614 & $\ldots$ & 18323 & $\ldots$ & 1800323 & $\ldots$  \\
				%				\hline
				%			\end{tabular} 
			%		\end{center}
		%	\end{example}

	Beck and Kifer \cite{Beck_Kifer} presented a formula for calculating the generalized Frobenius number for special cases. We give an alternative proof of their result using Theorem \ref{Main-Thm} as follows.
	
	\begin{corollary}\label{Cor-2<=m1 < m2 < m3}
		\cite[Theorem 1 ($k=3$ and $t=n$)]{Beck_Kifer}
		Let $m_1,m_2,m_3$ be pairwise coprime numbers and let $a_{1}=m_2 m_3, a_{2}=m_1 m_3, a_{3}= m_1 m_2$.
		Then, for $n\in\mathbb{Z}_{\geq0}$,
		\begin{align*}
			g(a,b,c; t_{n}) &= \operatorname{lcm}(a,b,c) (n+2) - a - b -c\\
			&= m_1 m_2 m_3 (n+2) - m_1 m_2 - m_1 m_3 - m_2 m_3 \,.
		\end{align*}
		Moreover, we have $\{ d(n;a_1, a_2, a_3) \mid n\geq 1\} = \{ t_k \mid k\in\mathbb{Z}_{\geq0}\}$.
	\end{corollary}
	\begin{proof}
		It is clearly that $d_{1} = \gcd(a_{2},a_{3}) = m_{1}, d_{2}=\gcd(a_{1},a_{3})=m_{2},$ and $d_{3} = \gcd(a_{1},a_{2}) = m_{3}$. Furthermore, $a_{1} \equiv 0\pmod{\frac{a_{2}}{m_{1}}}$, $a_{2} \equiv 0\pmod{\frac{a_{3}}{m_{2}}}$, and $a_{3} \equiv 0\pmod{\frac{a_{1}}{m_{3}}}$. Moreover, for $n\in\mathbb{Z}_{\geq 0}$ we have 
		\begin{equation*}
			\sum_{j=0}^{n}\left\lceil\frac{j a_{2}a_{3}}{a_{1}d_{1}^{2}} \right\rceil
			=
			\sum_{j=0}^{n}\left\lceil\frac{j a_{1}a_{3}}{a_{2}d_{2}^{2}} \right\rceil
			=
			\sum_{j=0}^{n}\left\lceil\frac{j a_{1}a_{2}}{a_{3}d_{3}^{2}} \right\rceil
			=
			\sum_{j=0}^{n} j
			=
			t_{n}.
		\end{equation*}
		So, by Theorem \ref{Main-Thm}, we obtain that for $n \in\mathbb{Z}_{\geq0}$,
		\begin{align*}
			g\big( 
			a_{1}, a_{2}, a_{3}; t_{n} \big)
			&=
			(n+1)m_{1}m_{2}m_{3} + m_{1}m_{2}m_{3} - m_{2}m_{3} - m_{1}m_{3} - m_{1}m_{2} 
			\\
			&=
			(n+2)m_{1}m_{2}m_{3} - m_{2}m_{3} - m_{1}m_{3} - m_{1}m_{2}
			\\
			&=
			\operatorname{lcm}(a,b,c) (n+2) - a - b -c.
		\end{align*}
		Moreover, by Theorem  1 of \cite{Beck_Kifer}, we obtain that 
        \begin{equation*}
        \{ d(n;a_1, a_2, a_3   ) \mid n\geq 1\} = \{ t_k \mid k\in\mathbb{Z}_{\geq0}\}.
        \tag*{\qedhere}
        \end{equation*}
        \end{proof}
	
	For example, let $(m_{1},m_{2},m_{3}) = (2,5,11)$ in Corollary \ref{Cor-2<=m1 < m2 < m3}. Then $a_{1} = 55, a_{2}=22$ and $a_{3}=10$. Then, for $n \in\mathbb{Z}_{>0}$, we compute $g(55,22,10; t_{n})$ by using Corollary \ref{Cor-2<=m1 < m2 < m3}. Hence, we have
		\begin{equation*}
			g(55,22,10;t_{n}) = 110(n+2) - 10 - 22 - 55= 110n + 133.
		\end{equation*}
		\begin{center}
			\begin{tabular}{|c|ccccccccc|}
				\hline
				$n  $ & 1 & 2 & 3 & 4 & 5 & 6 & $\ldots$ & 100 & $\ldots$   \\
				\hline
				$t_n$ & 1 & 3 & 6 & 10 & 15 & 21 & $\ldots$ & 5050 & $\ldots$  \\
				\hline
				$g(55,22,10;t_n)$ & 243 & 353 & 463 & 573 &683 &793 & $\ldots$ & 11133 & $\ldots$ \\
				\hline
			\end{tabular} 
		\end{center}
		Furthermore, we obtain that $\{ d(n;55,22,10) \mid n\geq 1\} = \{ t_k \mid k\geq 0\}$.

	The next special case is when one of the components is 1 and the others are arbitrary.
	\begin{corollary}\label{Cor-formula (1,a,b)}
		Let $a,b \in \mathbb{Z}_{>0}$ and $t \in \mathbb{Z}_{\geq 0}$	. Then 
		\begin{equation*}
			g\Big(1,a,b; \sum_{j=0}^{t}\left\lceil \frac{jb}{a}\right\rceil \Big)
			= tb-1
		\end{equation*}
	\end{corollary}
	\begin{proof}
		Recall that, in Section \ref{Section-ProofThm1}, we show that the equation \eqref{eq-relation} holds, that is
		\begin{equation}\label{eq: relation give g(1,a,b)}
			g\Big( a_{1}, \frac{a_{2}}{d_{1}}, \frac{a_{3}}{d_{1}}; \sum_{j=0}^{t}\left\lceil\frac{ja_{2}a_{3}}{a_{1}d_{1}^{2}} \right\rceil \Big) 
			= 
			g\Big( \frac{a_{2}}{d_{1}}, \frac{a_{3}}{d_{1}} ;t\Big). %\tag*{\eqref{eq-relation}}
		\end{equation} 
		At the beginning of the proof of Theorem \ref{Main-Thm}, we assumed that $a_{1} \equiv 0 \pmod{\frac{a_{2}}{d_{1}}}$. If $d_{1} = 1$, then $a_{1} = ka_{2}$ for some $k \in\mathbb{Z}_{>0}$. Then, by applying Lemma \ref{lemma-from-Beck-Kifer} on the left-hand side of \eqref{eq: relation give g(1,a,b)}, we obtain that
		\begin{align*}
			g\Big( ka_{2}, a_{2}, a_{3}; \sum_{j=0}^{t}\left\lceil\frac{ja_{3}}{k} \right\rceil \Big) 
			&=
			a_{2}g\Big(k, 1, a_{3};  \sum_{j=0}^{t}\left\lceil\frac{ja_{3}}{k} \right\rceil \Big) + a_{3}(a_{2}-1)
			\\
			&=
			a_{2}g\Big(1, k, a_{3};  \sum_{j=0}^{t}\left\lceil\frac{ja_{3}}{k} \right\rceil \Big)  + a_{2}a_{3} - a_{3}.	
		\end{align*}
		Applying \eqref{g(a,b;p)} to the right-hand side of \eqref{eq: relation give g(1,a,b)}, we have
		\begin{equation*}
			g( a_{2}, a_{3} ;t) = (t+1)a_{2}a_{3}- a_{2}-a_{3}.
		\end{equation*}
		Therefore,
		\begin{equation*}
			a_{2}g\Big(1, k, a_{3};  \sum_{j=0}^{t}\left\lceil\frac{ja_{3}}{k} \right\rceil \Big)  + a_{2}a_{3} - a_{3}
			=
			(t+1)a_{2}a_{3}- a_{2}-a_{3}.
		\end{equation*}
		Hence
		\begin{equation*}
			g\Big(1, k, a_{3};  \sum_{j=0}^{t}\left\lceil\frac{ja_{3}}{k} \right\rceil \Big)
			=
			ta_{3}-1.
		\end{equation*}
		Since $k$ and $a_{3}$ are arbitrary, thus for $a,b \in\mathbb{Z}_{>0}$
		\begin{equation*}
			g\Big(1, a, b;  \sum_{j=0}^{s}\left\lceil\frac{jb}{a} \right\rceil \Big) 
			=
			tb-1,
		\end{equation*} 
		as desired.
	\end{proof}
	\begin{example}
		Let $(a,b) = (4,9)$. Compute $ \displaystyle g\Big(1,4,9; \sum_{j=0}^{t}\left\lceil \frac{9j}{4} \right\rceil \Big)$ by using Corollary \mbox{\ref{Cor-formula (1,a,b)}\strut}. The result are shown as follows:
		
		\begin{center}
			\begin{tabular}{|c|ccccccccccc|}
				\hline
				$t$ & 0 & 1 & 2 & 3 & 4 & 5 & $\ldots$ & 100 & $\ldots$ & $10^{4}$ & $\ldots$   \\
				\hline
				$\Sigma_1 :=\displaystyle\sum_{j=0}^{t}\left\lceil \frac{9j}{4} \right\rceil$ & 0 & 3 & 8 & 15 & 24 & 36 & $\ldots$ & 11400 & $\ldots$ & 112515000 & $\ldots$ \\
				\hline
				$g(1,4,9;\Sigma_1)$ & -1 & 8 & 17 & 26 & 35 & 44 & $\ldots$ & 899 & $\ldots$ & 89999 & $\ldots$\\
				\hline
			\end{tabular} 
		\end{center}
		
		Similarly, if $(a,b) = (9,4)$, by using Corollary \mbox{\ref{Cor-formula (1,a,b)}\strut} the $ \displaystyle g\Big(1,9,4; \sum_{j=0}^{t}\left\lceil \frac{4j}{9} \right\rceil \Big)$  are as follows:
		
		\begin{center}
			\begin{tabular}{|c|ccccccccccc|}
				\hline
				$t$ & 0 & 1 & 2 & 3 & 4 & 5 &$\ldots$ & 100 & $\ldots$ & $10^{4}$ & $\ldots$  \\
				\hline
				$\Sigma_2 :=\displaystyle\sum_{j=0}^{t}\left\lceil \frac{4j}{9} \right\rceil$ & 0 & 1 &2&4&6&9&$\ldots$&2289&$\ldots$&22228889&$\ldots$ \\
				\hline
				$g(1,9,4;\Sigma_2)$ & -1&3&7&11&15&19&$\ldots$& 399 &$\ldots$& 39999 &$\ldots$
				\\ \hline
			\end{tabular} 
		\end{center}
	\end{example}
	
	%	\begin{example}
		%		Let $(a,b) = (7,17)$. Compute $ \displaystyle g\big(1,7,17; \sum_{j=0}^{s}\left\lceil \frac{17j}{7} \right\rceil \big)$ by using Corollary \ref{Cor-formula (1,a,b)}. The result are shown as follows:
		%		
		%		\begin{center}
			%			\begin{tabular}{|c|ccccccccccc|}
				%				\hline
				%				s & 0 & 1 & 2 & 3 & 4 & 5 & $\ldots$ & 100 & $\ldots$ & $10^{4}$ & $\ldots$ \\
				%				\hline
				%				$\Sigma :=\displaystyle\sum_{j=0}^{s}\left\lceil \frac{17j}{7} \right\rceil$ & 0& 3& 8&16&26& 39& $\ldots$ & 12307 & $\ldots$ & 121445000& $\ldots$ \\
				%				\hline
				%				$g(1,7,17;\Sigma)$ & -1& 16& 33& 50& 67& 84& $\ldots$&  1699 & $\ldots$ & 169999& $\ldots$ \\
				%				\hline
				%			\end{tabular} 
			%		\end{center}
		%		
		%		Similarly, if we let $(a,b) = (17,47)$, then we compute $ \displaystyle g\big(1,17,7; \sum_{j=0}^{s}\left\lceil \frac{7j}{17} \right\rceil \big)$ by using Corollary \ref{Cor-formula (1,a,b)}. We have that
		%		
		%		\begin{center}
			%			\begin{tabular}{|c|ccccccccccc|}
				%				\hline
				%				s & 0 & 1 & 2 & 3 & 4 & 5 & $\ldots$ & $100$ & $\ldots$ & $10^{4}$ & $\ldots$ \\
				%				\hline
				%				$\Sigma :=\displaystyle\sum_{j=0}^{s}\left\lceil \frac{7j}{17} \right\rceil$ & 0& 1& 2& 4& 6& 9& $\ldots$& 2127& $\ldots$& 20595000& $\ldots$\\
				%				\hline
				%				$g(1,17,7;\Sigma)$ & -1& 6 &13 &20 &27& 34& $\ldots$& 699& $\ldots$& 69999& $\ldots$
				%				\\ \hline
				%			\end{tabular} 
			%		\end{center}
		%	\end{example}
	%============================================

\end{document}